\documentclass[12pt]{amsart}
\usepackage{amsmath, amssymb, color}
\usepackage{amsthm}
\usepackage{color}
\usepackage{multirow,bigdelim}
\usepackage{graphicx}
\usepackage{array}
\usepackage{ifpdf}
\usepackage{epsfig}
\usepackage{amscd}
\usepackage{graphics}
\numberwithin{equation}{section}
\theoremstyle{definition}
\newtheorem{thm}{Theorem}[section]
\newtheorem{cor}[thm]{Corollary}

\newtheorem{lem}[thm]{Lemma}

\newtheorem{defn}[thm]{Definition}

\newtheorem{rem}[thm]{Remark}
\newtheorem*{acknowledgments}{Acknowledgments}

\newcommand{\Z}{\mathbb Z}
\newcommand{\C}{\mathbb C}
\newcommand{\ve}{\varepsilon}
\newcommand{\Kmn}{K^{m,n}}
\newcommand{\pK}{\pi_1(S^3 \setminus {\rm int}N(K))}
\newcommand{\pKmn}{\pi_1(S^4 \setminus {\rm int}N(K^{m,n}))}
\newcommand{{\punc}}{{\rm punc}}
\newcommand{{\SU}}{SU(2,\mathbb{C})}
\title{Irreducible $SL(2,\C)$-metabelian representations of branched twist spins}
\author{Mizuki Fukuda}
\address{Mathmatical Institute, Tohoku University, Sendai, 980-8578, Japan}
\email{fukuda.mizuki.r5@dc.tohoku.ac.jp}

\begin{document}
\begin{abstract}
An $(m,n)$-branched twist spin is a fibered $2$-knot in $S^4$ which is determined 
by a $1$-knot $K$ and coprime integers $m$ and $n$. 
For a $1$-knot, Nagasato proved that 
the number of conjugacy classes of irreducible $SL(2,\C)$-metabelian representations of the knot group of a $1$-knot 
is determined by the knot determinant of the $1$-knot.
In this paper, we prove that the number of irreducible $SL(2,\C)$-metabelian representations of 
the knot group of an $(m,n)$-branched twist spin up to conjugation is determined by the determinant of a $1$-knot in the orbit space 
by comparing a presentation of the knot group of the branched twist spin with the Lin's presentation of the knot group of the $1$-knot.
\end{abstract}
\maketitle
\section{Introduction}

A $2$-knot is a smoothly embedded $2$-sphere in $S^4$. A $2$-knot is said to be {\it fibered} if its complement admits a fibration structure over the circle with some natural structure in a tubular neighborhood of the $2$-knot. Although it is very difficult to see how the $2$-knot is embedded in $S^4$, the idea of admitting a fibration helps us to construct many examples of $2$-knots, such as spun knots, twist spun knots, rolling, deformed spun knots, branched twist spins, fibered homotopy-ribbon knots, etc~\cite{A,C,H,LM,Li,Pl2,Z}. A branched twist spin is a 2-knot which admits an $S^1$-action in its exterior. The terminology ``branched twist spin" appears in the book of Hillman~\cite{H}. It is known by Pao and Plotnick that a fibered $2$-knot is a branched twist spin if and only if its monodromy is periodic~\cite{Pl}. Therefore, this class has special importance among other known classes of fibered $2$-knots. Note that spun knots and twist spun knots are included in the class of branched twist spins.

We give here a short introduction of branched twist spins based on the classification of locally smooth $S^1$-actions on the $4$-sphere. 
Montgomery and Yang showed that effective locally smooth $S^1$-actions are classified into 
 four types~\cite{MY} and
Fintushel and Pao showed that 
there is a bijection between orbit data and weak equivalence classes of $S^1$-actions on $S^4$~\cite{Fi, Pa}.
Suppose that $S^1$ acts locally smoothly and effectively on $S^4$ and the orbit space is $S^3$.
Then there are at most two types of exceptional orbits called $\Z_m$-type and $\Z_n$-type, where $m,n$ are coprime positive integers.  
Let $E_m$ (resp. $E_n$) be the set of exceptional orbits of $\Z_m$-type (resp. $\Z_n$-type) and $F$ be the fixed point set.
The image of the orbit map of $E_n$, denoted by $E^{\ast}_n$, is an open arc in the orbit space $S^3$,
 and that of $F$, denoted by $F^{\ast}$, is the two points in $S^3$ which are the end points of $E^{\ast}_n$.
It is known that $E^{\ast}_m \cup E^{\ast}_n \cup F^{\ast}$ constitutes a $1$-knot $K$ in $S^3$ 
 and $E_n \cup F$ is diffeomorphic to the 2-sphere.  
The {\it $(m,n)$-branched twist spin} of $K$ is defined as $E_n \cup F$. 
Note that the $(m,1)$-branched twist spin is the $m$-twist spun knot and the $(0,1)$-branched twist spin is the spun knot.
If $K$ is a torus knot or a hyperbolic knot then its $(m , n)$-branched twist spins with $m>n$ and $m \geq 3$ are non-trivial. This follows from the fact that $K^{m , n}$ is not reflexive known by Hillman and Plotnick~\cite{HP}.

An oriented $k$-knot $K$ is said to be equivalent to another oriented $k$-knot $K^{\prime}$,
denoted by $K \sim K^{\prime}$, if there exists a smooth isotopy $H_t: S^{k+2}\to S^{k+2}$ such that $H_0={\rm id}$ and $H_1(K)=K^{\prime}$ as oriented $k$-knots.
In~\cite{F}, the author studied the elementary ideal of the fundamental group of the complement of a branched twist spin and gave a criterion to detect if two branched twist spins $K_1^{m_1,n_1}$ and $K_2^{m_2,n_2}$ are inequivalent by the knot determinants 
$\Delta_{K_1}(-1)$ and $\Delta_{K_2}(-1)$, where $\Delta_K (t)$ is the Alexander polynomial of a $1$-knot $K$ in $S^3$. 
Note that the definition of an $(m,n)$-branched twist spin is generalized to $(m,n) \in \Z \times \mathbb{N}$ by taking orientations into account, see Section~$2.1$. 

 The knot determinant is related to the number of irreducible $SL(2,\C)$-metabelian representations of the fundamental group of the knot complement~\cite{L,NY}. The aim of this paper is to count the number of such representations for a branched twist spin.
Similar to the results in~\cite{L,NY}, the number of such representations is given by the knot determinant as follows:

\begin{thm}\label{Main}
The number of irreducible $SL(2,\C)$-metabelian representations of $\pKmn$ is
\begin{equation*}
\left\{
\begin{split}
&\frac{|\Delta_K(-1)| -1}{2}& \hspace{20pt}(m : {\rm even})\\
&0&(m : {\rm odd}),
\end{split}
\right.
\end{equation*}
where $N(K^{m,n})$ is a compact tubular neighborhood of $K^{m,n}$ in $S^4$.
\end{thm}

As an immediate corollary, we obtain the same criterion as in~\cite{F}.

\begin{cor}[F.~\cite{F}]
Branched twist spins $K^{m_1,n_1}$ and $K^{m_2,n_2}$ are inequivalent if one of the following holds:
\begin{itemize}
\item[(1)] $m_1$ and $m_2$ are even and $|\Delta_{K_1}(-1)| \neq |\Delta_{K_2}(-1)|$,
\item[(2)] $m_1$ is even, $m_2$ is odd and $|\Delta_{K_1}(-1)| \neq 1$.
\end{itemize}
\end{cor}

This paper is organized as follows:
In Section 1, we define an $(m,n)$-branched twist spin $K^{m,n}$ as an oriented $2$-knot and introduce Plotnick's presentation of $\pKmn$.
In Section 2, we state the Lin's presentation of a $1$-knot and the Nagasato-Yamaguchi's presentation of the $m$-fold cyclic branched cover of $S^3$ along $K$.
In Section 3, we observe irreducible $SL(2,\C)$-metabelian representations of $\pKmn$ 
and  prove Theorem~\ref{Main}.

\begin{acknowledgments}
The author is grateful to his supervisor, Masaharu Ishikawa, for many helpful suggestions. 
\end{acknowledgments}

\section{Two presentations of branched twist spins}

\subsection{The $(m,n)$-branched twist spin}
Suppose that $S^4$ has an effective locally smooth $S^1$-action.
Let $E_m$ be the set of exceptional orbits of $\mathbb{Z}_m$-type, where $m$ is a positive integer, 
and $F$ be the fixed point set.
Set $E^{\ast}_m$ and $F^{\ast}$ to be the image of $E_m$ and $F$ by the orbit map, respectively.
Montgomery and Yang showed that effective locally smooth $S^1$-actions are classified into 
the following four types:
(1) $\{D^3\}$,(2) $\{S^3\}$, (3) $\{S^3 , m\}$, (4) $\{ (S^3, K), m ,n\}$, which are called orbit data~\cite{MY}. 
The $3$-ball and the $3$-sphere in these notations represent the orbit spaces.
In case (4), the union $E^{\ast}_m \cup E^{\ast}_n \cup F^{\ast}$ constitutes a $1$-knot $K$ in the orbit space $S^3$ 
and the union $E_n \cup F$ is diffeomorphic to the $2$-sphere.
This $2$-sphere is embedded in $S^4$, and is called the $(m,n)$-branched twist spin of $K$, denoted by $K^{m,n}$.
In case (3), for an arc $A^{\ast}$ in $S^3$ whose end points are $F^{\ast}$, the preimage of $A^{\ast}$ is denoted by $A$.
Then the union $A \cup F$ is diffeomorphic to the $2$-sphere, 
and is called a twist spun knot.
We may regard an $m$-twist spun knot as $K^{m,1}$, where $K$ is $A^{\ast} \cup E^{\ast}_m \cup F^{\ast}$.

We recall the definition of $(m,n)$-branched twist spins for $(m,n)\in\mathbb{Z}\times\mathbb{N}$ in~\cite{F}.
First, we remark that the definition in ~\cite{F} depends on the choice of the orientation of $K$.
Actually, in the definition we fixed a preferred meridian-longitude  $(\theta, \phi)$ of $S^3 \setminus {\rm int}N(K)$, 
where $N(K)$ is a compact tubular neighborhood of $K$, 
and replacing $(\theta, \phi)$ by $(-\theta, -\phi)$ may change the equivalence class of $\Kmn$.

We give the definition of $\Kmn$.
Let $K$ be a $1$-knot in $S^3$ and $(m,n)$ be a pair of integers in $(\mathbb{Z} \setminus \{0\}) \times\mathbb N$
such that $|m|$ and $n$ are coprime. 
We decompose the orbit space $S^3$ into five pieces as follows:
$$
S^3 = (S^3 \setminus {\rm int}N(K))  \cup (E^{c\ast}_m \times D^2) \cup (E^{c\ast}_n \times D^2) \cup (D^{3\ast}_1\sqcup D^{3\ast}_2),
$$
where $D^{3\ast}_1 \sqcup D^{3\ast}_2$ is a compact neighborhood of $F^{\ast}$ and $E^{c\ast}_m$ and $E^{c\ast}_n$ are the connected components of $K \setminus \text{int}(D^{3\ast}_1 \cup D^{3\ast}_2)$ 
such that  $E_m^{c\ast}\subset E_m^{\ast}$ and $E_n^{c\ast}\subset E_n^{\ast}$, see Figure~\ref{test}.
Considering the preimage of the orbit map, we decompose $S^4$ as follows:
\begin{equation}\label{dec}
S^4  = (\left(S^3 \setminus {\rm int} N(K) \right)\times S^1) \cup (E_m \times D^2) \cup (E_n \times D^2) \times (D^4_1 \sqcup D^4_2).
\end{equation}

Let $p$ denote the orbit map.
Choosing a point $z^{\ast}_m$ in $E^{c\ast}_m$, let $D^{2\ast}_{z^{\ast}_m}$ be a 2-disk in $S^3$ centered at $z^{\ast}_m \in E^{c\ast}_m$ and transversal to $E^{c\ast}_m$.
The preimage $p^{-1}(D^{2\ast}_{z^{\ast}_m})$ is a solid torus $V_m$ whose core is the exceptional orbit of $\mathbb{Z}_m$-type.

\begin{figure}[htbp]
\centering
\includegraphics[scale=0.5]{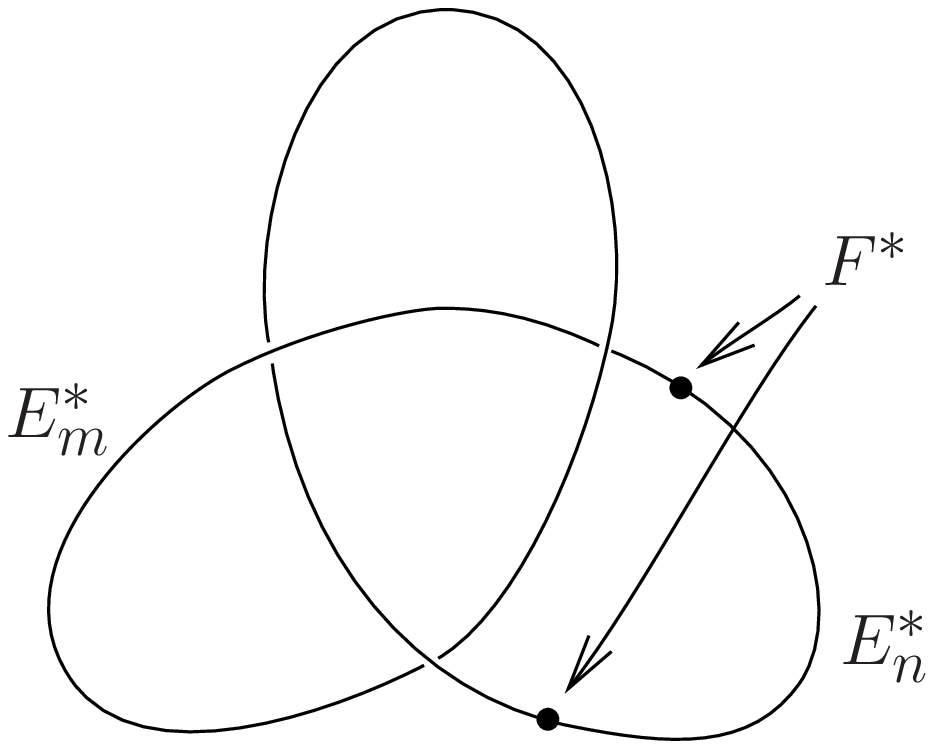}
\hspace{50pt}
\includegraphics[scale=1]{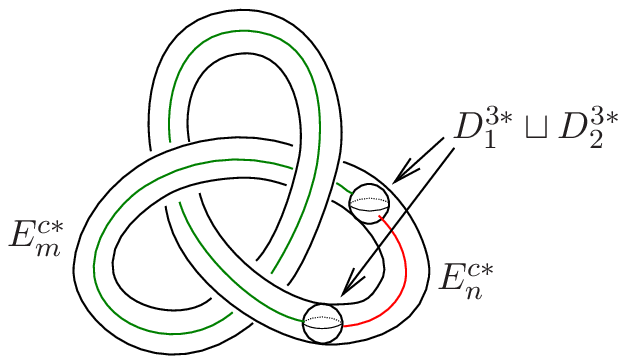}
\caption{Decomposition of $S^3$}\label{test}
\end{figure}

Now we discuss the orientations of $V_m$ and $E^{c\ast}_m$.
Let $K$ be an oriented $1$-knot in $S^3$.
First, fix the orientation of $S^4$ and those of orbits such that they coincide with the direction of the $S^1$-action.
These orientations determine the orientation of $V_m\times E_m^{c*}$.
Let  $(\theta,\phi)$ be the preferred meridian-longitude pair of $K$ such that the orientation of the longitude $\phi$ coincides the orientation of $K$.
From the decomposition~(\ref{dec}),
we can see that $\phi$ is regarded as a coordinate of the second factor of $V_m\times E^{c\ast}_m$.
We assign the orientation of $V_m$ so that the orientation of $V_m\times E^{c\ast}_m$ coincides with the given orientation of $S^4$.
Finally, we choose the meridian and longitude pair $(\Theta,H)$ 
of $V_m\cong D^2\times S^1$ such that $H$ becomes the meridian
of $V_n$ in the decomposition $V_m\cup V_n=p^{-1}(\partial D_i^{3*})$
and the orbits of the $S^1$-action are in the direction 
$\ve n\Theta+|m|H$ with $n>0$, where $\ve= 1$ if $m\geq 0$ and $\ve = -1$ if $m<0$. 

\begin{defn}[Branched twist spin]
Let $K$ be an oriented knot in $S^3$.
For each pair $(m,n)\in\mathbb{Z}\times \mathbb{N}$ with 
$m\ne 0$ such that $|m|$ and $n$ are coprime,
let $K^{m,n}$ denote the $2$-knot $E_n\cup F$.
If $(m,n)=(0,1)$ then define $K^{0,1}$ to be the spun knot of $K$.
The $2$-knot $K^{m,n}$ is called an $(m,n)$-$branched\ twist\ spin$ of $K$.
\end{defn}
Note that the branched twist spin $K^{m,1}$ constructed from $\{(S^3 , K),m,1\} $ is an $m$-twist spun knot of $K$.
\begin{rem}\label{rem}
Let $-K$ be an oriented knot obtained from $K$ by reversing the orientation of $K$.
From the construction of $\Kmn$, we see that $K^{m,n}$ is equivalent to $-(-K)^{-m,n}$.
\end{rem}

Let $K$ be a $k$-knot in $S^{k+2}$. 
The fundamental group of the knot complement $S^{k+2}\setminus \text{int}N(K)$ is called the {\it knot group} of $K$, 
where $N(K)$ is a compact tubular neighborhood of $K$.

\begin{lem}[\cite{F}]\label{P}
Let $K$ be an oriented $1$-knot and $K^{m,n}$ be the $(m,n)$-branched twist spin of $K$ with $(m,n)\in\mathbb Z\times \mathbb N$, where $|m|$ and $n$ are coprime.
Let $\langle y_1,\ldots, y_s\mid r_1,\ldots,r_t\rangle$
be a presentation of the knot group of $K$ such that $y_1$ is a meridian.
Then the knot group of $K^{m,n}$ has the presentation
\begin{equation}\label{R1}
\pi_1(S^4 \setminus \text{int}N(K^{m , n})) \cong 
\langle y_1,\ldots,y_s,h\ |\ r_1,\ldots,r_t,y_ihy^{-1}_ih^{-1}, y_1^{|m|} h^{\beta} \rangle,
\end{equation} 
where $\beta$ is an integer such that 
$n\beta\equiv \ve$ (mod $m$) if $m$ is non-zero and $\beta = 1$ if $m=0$.
Recall that $\ve = 1$ if $m\geq 0$ and $\ve = -1$ if $m<0$.
\end{lem}
Note that $\pKmn$ is isomorphic to $\pi_1(S^4 \setminus \text{int}N((-K)^{-m , n}))$ by Remark~\ref{rem}.

\subsection{Plotnick's presentation}
Assume that $m \neq 0$.
We ignore the orientation of $\Kmn$ since we are interested in the fundamental group of its complement.
By Remark~\ref{rem}, changing the orientation of $K$ and the sign of $m$ if necessary, we can assume that $m$ is positive.
Pao constructed the knot complement of $\Kmn$ as follows~\cite{Pa}:
Let $M_K$ be the $m$-fold cyclic branched cover of $S^3$ along $K$ and $\tau : M_K \to M_K$ be the diffeomorphism associated with 
the canonical deck transformation of $M_K$.
Let $M_K \times_{\tau^n} S^1$ be the manifold obtained from $M_K \times I$ by identifying $M_K \times \{0\}$ with $M_K \times \{1\}$
by $(z , 1) \mapsto (\tau^{ n} z , 0)$, where $\tau^{n}$ means the $n$-th power of composite of $\tau$.
Note that $M_K \times_{\tau^{n}} S^1$ has the natural $S^1$-action $\varphi_s\langle y , t \rangle = \langle y , t+ s \rangle$,
where $\langle y , t \rangle$ denotes the image of $(y , t) \in M_K \times I$ by the identification.
Let $x$ be a branch point of $M_K$.
Then the orbit of $\langle x ,0 \rangle$ is a circle in $M_K \times_{\tau^{ n}} S^1$.
There is a neighborhood of the orbit which is invariant by the $S^1$-action, denoted by $T$. 
It is known in~\cite{Pa} that the knot complement of $\Kmn$ is diffeomorphic to  
($M_K \times_{\tau^{ n}} S^1)\setminus {\rm int}T$, which is also diffeomorphic to $\punc (M_K) \times_{\tau^{n}} S^1$, where 
$\punc(M_K) = M_K \setminus {\rm int} D^3$ with $D^3$ being a $3$-ball in $M_K$.
Note that $K^{m,n}$ is regarded as the branch set of the $n$-fold cyclic branched cover of $S^4$ along the $m$-twist spun knot of $K$.

The following lemma is shown by Plotnick in~\cite{Pl1}.
\begin{lem}[Plotnick~\cite{Pl1}]\label{Pl}
Let $\Kmn$ be a branched twist spin of $K$.
Then the following holds:
$$
\pKmn \cong
\pi_1(\punc (M_K)) \ast \langle \eta \rangle/ \langle \eta (\tau^{n} z) \eta^{-1} = z\ {\rm for\ all}\ z \in \pi_1(M_K) \rangle,
$$ 
where $\eta$ is a meridian of $\Kmn$.
\end{lem}

\subsection{Lin's presentation}
Let $K$ be a $1$-knot in $S^3$. 
A Seifert surface $S$ of $K$ is called {\it free} if $S^3 = N(S) \cup (S^3\setminus {\rm int}N(S))$ gives a Heegaard splitting of $S^3$. 
It is known that any $1$-knot has a free Seifert surface.
A presentation of $\pK$ is obtained from the Heegaard splitting associated to a free Seifert surface as follows:
Let $S$ be a free Seifert surface of $K$ of genus $g$ and $W$ be a spine of $S$.
Then $H_1 = S \times [-1,1]$ and $H_2 = S^3 \setminus {\rm int}H_1$ is a Heegaard splitting of $S^3$.
Let $K^{\prime}$ be a simple closed curve obtained from $K$ by pushing it into $H_1$ slightly.
Choose a base point $\ast$ in $W\subset S \times \{0\}$ such that $\ast$ does not on $K$ and $K^{\prime}$. 
Since $H_1$ and $H_2$ are handlebodies with genus $2g$,
we may choose generators $a_1, \ldots , a_{2g}$ of $\pi_1(H_1)$ and 
generators $x_1, \ldots , x_{2g}$ of $\pi_1(H_2)$.
Let $a^{+}_1, \ldots, a^{+}_{2g}$ and $a^{-}_1, \ldots, a^{-}_{2g}$ denote the loops 
$a_1 \times \{1\}, \ldots, a_{2g}\times \{1\}$ and $a_1 \times \{-1\}, \ldots ,a_{2g}\times \{-1\}$.
Each $a^{+}_i$ (resp. $a^{-}_i$) is written in a word of $x_1, \ldots, x_{2g}$ by the homeomorphism from $\partial H_2$ to $\partial H_1$.
The words of $a^{+}_i$ (resp. $a^{-}_i$) are denoted by $\alpha_i$ (resp. $\beta_i$) for $i = 1,\ldots, 2g$.
There is a unique arc $c$, up to isotopy, such that $(\ast \times [-1,1]) \cup c$ is a meridian of $K^{\prime}$.
The homotopy class of this loop is denoted by $\mu$. 
From van Kampen theorem, the following theorem holds:
\begin{lem}[Lin~\cite{L}]\label{L}
Let $K$ be a $1$-knot in $S^3$ and $S$ be a free Seifert surface of $K$.
Let $S^3 = H_1 \cup H_2$ be the Heegaard splitting associated to $S$.
For generators $x_1,\ldots, x_{2g}$ of $\pi_1(H_2)$, $\pi_1(E_K)$ has the following presentation:
\begin{equation}\label{L1}
\pK \cong \langle x_1 , \ldots , x_{2g}, \mu\ |\  \mu \alpha_i \mu^{-1} = \beta_i \rangle,
\end{equation}
where $g$ is the genus of $S$, and 
$\alpha_i$,  $\beta_i$ are the words in $x_1,\ldots,x_{2g}$ determined above.
\end{lem}
\noindent

Let $\langle x_1 , \ldots , x_{2g}, \mu\ |\  \mu \alpha_i \mu^{-1} = \beta_i \rangle$ be a Lin's presentation of $\pK$.
Denote the sum of indices of $x_j$ in $\alpha_i$ by $v_{ij}$ and that in $\beta_i$ by $u_{ij}$.
Then the $2g \times 2g$ matrix $V = (v_{ij})$ is defined.
The matrix V is called a Seifert matrix and $\det (V + {^t}V)$ is called the knot determinant of $K$, which equals to $\Delta_K(-1)$.
Note that all generators $x_1,\ldots ,x_{2g}$ are commutators of $\pK$.
Let $\rho_0 : \pK \to SL(2,\mathbb{C})$ be an $SL(2,\mathbb{C})$-metabelian representation.
Since all $x_1,\ldots ,x_{2g}$ are commutators of $\pK$, we can assume that 

\begin{equation}\label{lineq2}
\rho_0(x_i) = 
\left(
\begin{matrix}
 \lambda_i & 0\\
0 & {\lambda_i}^{-1}
\end{matrix}
\right)
\hspace{10pt}
{\rm or}
\hspace{10pt}
\pm \left(
\begin{matrix}
1 & \omega_i \\
0 & 1
\end{matrix}
\right) 
\ \ \ \ 
(i = 1,\ldots, 2g),
\end{equation}
up to conjugation.
If there exists $i\in\{1,\ldots, 2g\}$ such that $\rho_0(x_i)$ is
$\pm \left(
\begin{matrix}
1 & \omega_i \\
0 & 1
\end{matrix}
\right)$,
then all $\rho(x_j)$ are of the forms
 $
 \pm \left(
\begin{matrix}
1 & \omega_j \\
0 & 1
\end{matrix}
\right).
 $  
This is implied by all $x_j$ are conjugate to each other and such a representation is abelian, especially reducible. 
 Therefore we can assume that 
 an irreducible $SL(2,\mathbb{C})$-metabelian representation satisfies

\begin{equation}\label{lineq}
\rho(x_i) = 
\left(
\begin{matrix}
 \lambda_i & 0\\
0 & {\lambda_i}^{-1}
\end{matrix}
\right),
\hspace{30pt}
\rho(\mu) = 
\left(
\begin{matrix}
0 & -1\\
1 & 0
\end{matrix}
\right) 
\ \ \ \ 
(i = 1,\ldots, 2g),
\end{equation}
up to conjugation. Since $\alpha_i$ and $\beta_i$ are written in words $x_1,\ldots,x_{2g}$, each $\rho_0(\alpha_i)$ and $\rho_0(\beta_i)$ is a diagonal matrix.
From (\ref{lineq}), Lin checked directly the number of irreducible $SU(2,\C)$-metabelian representations of $\pK$. 

\begin{thm}[Lin~\cite{L}]\label{L2}
The number of conjugacy classes of irreducible $SU(2,\mathbb{C})$-metabelian representations of $\pK$ is 
$$
\frac{|\Delta_K(-1)| -1}{2}.
$$
\end{thm}

\begin{rem}
In ~\cite{N}, Nagasato showed that the same statement holds for irreducible $SL(2,\mathbb{C})$-metabelian representations of $\pK$. 
\end{rem}

\subsection{Nagasato-Yamaguchi's presentation}

Let $M_K$ be the $m$-fold cyclic branched cover of $S^3$ along $K$ and $\tau$ be the canonical deck transformation on $M_K$. 
The fundamental region of $M_K$ contains a free Seifert surface of $K$.
Nagasato and Yamaguchi gave a presentation of $\pi_1(M_K)$ from the Lin's presentation of $\pK$.
\begin{thm}[Nagasato,Yamaguchi~\cite{NY}]\label{NY}
Let $ \langle x_1 , \ldots , x_{2g}, \mu\ |\  \mu \alpha_i \mu^{-1} = \beta_i \rangle$ be a Lin's presentation of a $1$-knot $K$.
Then $\pi_1(M_K)$ has the following presentation:
$$
\pi_1(M_K) \cong 
\langle \tau^{0}\tilde{x}_1 , \ldots , \tau^{0}\tilde{x}_{2g} ,\ldots , \tau^{m-1}\tilde{x}_1 , \ldots , \tau^{m-1}\tilde{x}_{2g} \ |\  \tilde{\alpha}^{(j)}_i = \tilde{\beta}^{(j-1)}_i \rangle,
$$
where $\tilde{x}_i$ is the lift of $x_i$ to $M_K$, and $\tilde{\alpha}^{(j)}_i, \tilde{\beta}^{(j)}_i$ are the words obtained from
$\alpha_i , \beta_i$ by replacing $x_1, \ldots x_{2g}$ with $\tau^j \tilde{x}_1, \ldots ,\tau^j \tilde{x}_{2g}$ 
for $i =1,\ldots ,2g$ and $j \equiv 0,\ldots ,m-1\ ({\rm mod}\ m)$.
\end{thm}

We can rewrite the presentation in Lemma~\ref{Pl}
by applying a Nagasato-Yamaguchi's presentation to $\punc(M_K)$ as follows:

\begin{equation}\label{R2}
\begin{split}
\langle \tau^{0}\tilde{x}_1 , \ldots , \tau^{0}\tilde{x}_{2g} ,\ldots , \tau^{m-1}\tilde{x}_1 , \ldots ,& \tau^{m-1}\tilde{x}_{2g} , \eta \ |\\  
 &\tilde{\alpha}^{(j)}_i = \tilde{\beta}^{(j-1)}_i , \eta \tau^{j+n}\tilde{x}_i \eta^{-1} = \tau^j \tilde{x}_i\rangle.
\end{split}
\end{equation}

\section{Proof of Theorem~\ref{Main}}

 We first introduce a property of irreducible $SL(2,\mathbb{C})$-metabelian representations of $\pKmn$ from~(\ref{R2}).
\begin{lem}\label{con}
Let $\rho$ be an irreducible $SL(2,\C)$-metabelian representation of $\pKmn$.
Then, up to conjugation, $\rho$ is of the form
$$
\rho(\tau^j\tilde{x}_i) =\left(
\begin{matrix}
 \lambda^{(j)}_i & 0\\
0 & {\lambda^{(j)}_i}^{-1}
\end{matrix}
\right)
 \hspace{20pt}
\rho(\eta) = 
\left(
\begin{matrix}
 0& -1\\
1 & 0
\end{matrix}
\right),
$$
where $i = 1,\ldots 2g , j \equiv 0,\ldots, m-1\ ({\rm mod}\ m)$, and $\lambda^{(j)}_i \neq {\lambda^{(j)}_i}^{-1}$ for some $i,j$.
\end{lem}
\begin{proof}
Since $\rho$ is a metabelian representation, $\rho([\pKmn , \pKmn])$ is an abelian group.
Up to conjugation of $\rho$, we can assume that $\rho(x)$ is a diagonal matrix for any $x \in [\pKmn , \pKmn]$.
Since the generators $\tau^j\tilde{x}_i$ are on the Seifert surface of $\Kmn$,
all $\tau^j\tilde{x}_i$ are commutators in $\pKmn$.
Then $\rho(\tau^j\tilde{x}_i)$ are of the forms
\begin{equation*}
\rho(\tau^j\tilde{x}_i) =\left(
\begin{matrix}
 \lambda^{(j)}_i & 0\\
0 & {\lambda^{(j)}_i}^{-1}
\end{matrix}
\right)
 \hspace{30pt}
(\lambda^{(j)}_i \in \mathbb{C}\setminus {0}),
\end{equation*}
see the observation of ($\ref{lineq}$).
The matrix $\rho(\eta)$ is determined by the relations $\eta \tau^{j+ n}\tilde{x}_i \eta^{-1} = \tau^{j}\tilde{x}_i$ as follows.
Set 
$
\rho(\eta) = \left(
\begin{matrix}
a & b\\
c &d
\end{matrix}
\right) \in SL(2,\C).
$
Then $\rho(\eta \tau^{j+ n}\tilde{x}_i)$ and $\rho(\tau^{j}\tilde{x}_i \eta)$ are given as
\begin{equation*}
\begin{split}
&\rho(\eta \tau^{j+ n}\tilde{x}_i) = 
\left(
\begin{matrix}
 a& b\\
c & d
\end{matrix}
\right)
\left(
\begin{matrix}
 \lambda^{(j+ n)}_i& 0\\
0 & {\lambda^{(j + n)}_i}^{-1}
\end{matrix}
\right)
=
\left(
\begin{matrix}
 a\lambda^{(j+n)}_i& b{{\lambda}^{(j+n)}_i}^{-1}\\
c\lambda^{(j+n)}_i & d{{\lambda}^{(j+n)}_i}^{-1}
\end{matrix}
\right),
\\
&\rho(\tau^{j}\tilde{x}_i \eta)=
\left(
\begin{matrix}
 \lambda^{(j)}_i& 0\\
0 & {{\lambda}^{(j)}_i}^{-1}
\end{matrix}
\right)
\left(
\begin{matrix}
 a& b\\
c & d
\end{matrix}
\right)
=
\left(
\begin{matrix}
 a\lambda^{(j)}_i& b\lambda^{(j)}_i\\
c{{\lambda}^{(j)}_i}^{-1} & d{{\lambda}^{(j)}_i}^{-1}
\end{matrix}
\right).
\end{split}
\end{equation*}
These two matrices must be the same.
Assume that $\lambda^{(j+ n)}_i = \lambda^{(j)}_i$ for all $i,j$.
Since $m$ and $n$ are coprime, $\lambda^{(j)}_i = \lambda^{(0)}_i$ for any $i,j$.
If $\lambda^{(0)}_i = {{\lambda}^{(0)}_i}^{-1}$ for any $i$, then 
$\rho(\tau^{(j)}\tilde x_i) =
\left(
\begin{matrix}
 \pm1& 0\\
0 & \pm1
\end{matrix}
\right)
$ for any $i,j$.
Then $\rho$ is not irreducible.
If $\lambda^{(0)}_i \neq {{\lambda}^{(0)}_i}^{-1}$ for some $i$, then
 $\rho(\eta)$ is a diagonal matrix and $\rho(\pKmn)$ becomes an abelian group.
It also contradicts the irreducibility of $\rho$.
Therefore $\lambda^{(j+n)}_i \neq \lambda^{(j)}_i$ for some $i,j$.
In this case, $a=d = 0$ and 
$\rho(\eta) = \left(
\begin{matrix}
0 & b\\
-b^{-1} &0
\end{matrix}
\right).$
Set $B=\left(
\begin{matrix}
 b^{\frac{1}{2}}& 0\\
0 & b^{-\frac{1}{2}}
\end{matrix}
\right).$
Since 
$B\rho(\eta)B^{-1}=
\left(
\begin{matrix}
 0& -1\\
1 & 0
\end{matrix}
\right)
$ 
and 
$B\rho(\tau^{j}\tilde{x}_i)B^{-1}=\rho(\tau^{j}\tilde{x}_i)$,
we have $\rho(\eta)=
\left(
\begin{matrix}
 0& -1\\
1 & 0
\end{matrix}
\right)
$
up to conjugation.
\end{proof}

Let $S_0$ be a free Seifert surface of $K$ contained in  a fundamental reagion of $M_K$, where $M_K$ is a fiber of $\Kmn$.
Let $S_1, \ldots, S_{m-1}$ be copies of $S_0$ by the deck transformations.
We want to know relation between $\lambda^{(j)}_{i}$ and $\lambda^{(j+1)}_{i}$ for $j \equiv 1,\ldots, m-1\ ({\rm mod}\ m)$.
The relation $\eta \tau^{j+ n}\tilde{x}_i \eta^{-1} = \tau^{j}\tilde{x}_i$ means that
the conjugation by $\eta$ brings $\tau^{j+ n}\tilde{x}_i$ on $S_{j+n}$ to $\tau^{j}\tilde{x}_i$ on $S_j$.
Let $q$ be an integer such that $nq\equiv 1 ({\rm mod}\ m)$ and 
take conjugation of $\tau^j\tilde x_i$ by $\eta^q$. Then we obtain the relation
\begin{equation}\label{rel}
\tau^{j}\tilde{x}_i = \eta \tau^{j+n}\tilde{x}_i \eta^{-1} = \eta^{q} \tau^{j+nq}\tilde{x}_i \eta^{-q} = \eta^q \tau^{j+1}\tilde{x}_i \eta^{-q},
\end{equation}
which brings $\tau^{j+1}\tilde{x}_i$ on $S_{j+1}$ to $\tau^{j}\tilde{x}_i$ on $S_{j}$, 
 where we used $nq\equiv 1\ ({\rm mod}\ m)$.

Let $\rho$ be an irreducible $SL(2,\C)$-metabelian representation of $\pKmn$ in Lemma~\ref{con}.
From the relation~(\ref{rel}) and Lemma~\ref{con},
\begin{equation}\label{case}
\left(
\begin{matrix}
 \lambda^{(j+1)}_i & 0\\
0 & {\lambda^{(j+1)}_i}^{-1}
\end{matrix}
\right)
=\left\{
\hspace{10pt}
\begin{split}
\left(
\begin{matrix}
 \lambda^{(j)}_i & 0\\
0 & {{\lambda}^{(j)}_i}^{-1}
\end{matrix}
\right)
\hspace{30pt}
(q : {\rm even})
\hspace{3pt}\\
\left(
\begin{matrix}
 {\lambda^{(j)}_i}^{-1} & 0\\
0 & \lambda^{(j)}_i
\end{matrix}
\right)
\hspace{30pt}
(q : {\rm odd}).
\end{split}
\right.
\end{equation}

Suppose that $m$ is even.
Then $q$ is odd since $m$ and $q$ are coprime.
We define the representation $\overline{\rho}$ by
 $$
 \overline{\rho}(x) = \rho(\eta x \eta^{-1}) =\rho(\eta)\rho(x)\rho(\eta^{-1})
 $$
for all $x \in \pKmn$.
Note that $\rho(\eta)=
\left(
\begin{matrix}
0 & -1\\
1 & 0
\end{matrix}
\right)
$
by Lemma~\ref{con}.
 In particular, $\overline{\overline{\rho}}(x) = \rho(x)$.
By~(\ref{case}), $\rho(\tau^{j+1}\tilde{x}_i) = \overline{\rho}(\tau^{j}\tilde{x}_i)$ for all $i,j$.
Since $\tilde{\alpha}^{(j)}_i$ and $\tilde{\beta}^{(j)}_i$ are words written in $\tau^{j}\tilde{x}_1, \ldots \tau^{j}\tilde{x}_{2g}$, 
we have 
\begin{equation}\label{alpha}
\rho(\tilde{\alpha}^{(j+1)}_i) = \overline{\rho}(\tilde{\alpha}^{(j)}_i),
\hspace{20pt} 
\rho(\tilde{\beta}^{(j+1)}_i) = \overline{\rho}(\tilde{\beta}^{(j)}_i).
\end{equation}
On the other hand, 
\begin{equation}\label{rel2}
\tilde{\beta}^{(j)}_i = \tilde{\alpha}^{(j+1)}_i =\eta^{-q} \tilde{\alpha}^{(j)}_i \eta^{q}  
\end{equation}
holds, where the relation~(\ref{rel}) is applied to the second equality. 
Since $q$ is odd,  $\rho(\eta^{q} x \eta^{-q}) = \rho(\eta^{q}) \rho(x) \rho(\eta^{-q}) =  \rho(\eta) \rho(x) \rho(\eta^{-1}) = \overline{\rho}(x)$. 
Hence, by~(\ref{rel2}), we have
\begin{equation}\label{albe}
 \rho(\tilde{\beta}^{(j)}_i)= \overline{\rho}(\tilde{\alpha}^{(j)}_i).
 \end{equation}
From~(\ref{alpha}) and~(\ref{albe}), one can see that the relations of representations of the first relations $\tilde{\alpha}^{(j)}_i = \tilde{\beta}^{(j-1)}_i$ in~(\ref{R2}) 
are equivalent to $\rho(\tilde{\beta}^{(0)}_i) = \rho(\eta \tilde{\alpha}^{(0)}_i \eta^{-1})$.
\ \\
\ \\

The second relations $\eta \tau^{j+n}\tilde{x}_i \eta^{-1} = \tau^{j}\tilde{x}_i$ in~(\ref{R2}) are equivalent to $\eta^q \tau^{j+1}\tilde{x}_i \eta^{-q} = \tau^{j}\tilde{x}_i$ for all $j$ as checked in ~(\ref{rel}).
Therefore  $\rho(\eta \tau^{j+n}\tilde{x}_i \eta^{-1}) = \rho(\tau^{j}\tilde{x}_i)$ are equivalent to $ \rho(\eta \tau^{j+1}\tilde{x}_i \eta^{-1}) =\rho(\tau^{j}\tilde{x}_i )$.
Hence the number of irreducible $SL(2,\C)$-representations of the presentation~(\ref{R2}) is equal to that of representations of the group presented by 
\begin{equation}\label{simp}
\langle \tau^{0}\tilde{x}_1 , \ldots , \tau^{0}\tilde{x}_{2g},\ldots, \tau^{m-1}\tilde{x}_1 , \ldots , \tau^{m-1}\tilde{x}_{2g}   , \eta \ |\  
\eta \tilde{\alpha}^{(0)}_i \eta^{-1} = \tilde{\beta}^{(0)}_i, \eta \tau^{j+1}\tilde{x}_i \eta^{-1} = \tau^{j}\tilde{x}_i\rangle.
\end{equation}
Now, we reduce the generators 
$\tau^{1}\tilde{x}_1 , \ldots , \tau^{1}\tilde{x}_{2g},\ldots, \tau^{m-1}\tilde{x}_1 , \ldots , \tau^{m-1}\tilde{x}_{2g}$ 
and the relations $\eta \tau^{j+1}\tilde{x}_i \eta^{-1} = \tau^{j}\tilde{x}_i$ from the above presentation 
to simplify counting the number of irreducible $SL(2,\C)$-metabelian representations of $\pKmn$.
 
\begin{lem}\label{alt}
Let $m$ be an even integer.
Then the number of irreducible $SL(2,\C)$-metabelian representations of $\pKmn$ coincides that of the group $G$ presented by
\begin{equation}\label{red}
\langle \tau^{0}\tilde{x}_1 , \ldots , \tau^{0}\tilde{x}_{2g}, \eta \ |\  
\eta \tilde{\alpha}^{(0)}_i \eta^{-1} = \tilde{\beta}^{(0)}_i\rangle.
\end{equation}
\end{lem}
\begin{proof}
A representation of~(\ref{simp}) is a representation of ~(\ref{red}).
So, we prove the converse. The representation of $\tau^j\tilde x_i$ for $j\equiv 1,\ldots,m-1\ ({\rm mod}\ m)$ is
determined by the equality 
$\rho(\tau^{j+1}\tilde{x}_i) = \overline{\rho}(\tau^{j}\tilde{x}_i)$ obtained from~(\ref{case}).
Hence, it is enough to prove that any irreducible $SL(2,\C)$-metamerian representation $\rho$ of (\ref{red}) has the property 
$\rho(\eta \tau^{0}\tilde{x}_i \eta^{-1}) = \overline{\rho}(\tau^{0}\tilde{x}_i)$.
Since the presentation in (\ref{red}) is exactly of the same form as the Lin's presentation~(\ref{L1}),
 all $\rho(\tau^{0}\tilde{x}_i)$ and $\rho(\eta)$ are of the forms
$$
\rho(\tau^{0}\tilde{x}_i)=
\left(
\begin{matrix}
 \lambda^{(0)}_i & 0\\
0 & {{\lambda}^{(0)}_i}^{-1}
\end{matrix}
\right),
\hspace{30pt}
\rho(\eta)=
\left(
\begin{matrix}
0 & -1\\
1 & 0
\end{matrix}
\right),
$$
for $i=1,\ldots, 2g$ up to conjugation, see~\cite{N}.
Then, by the definition of $\overline{\rho}$,
$$
\overline{\rho}(\tau^{0}\tilde{x}_i)=
\left(
\begin{matrix}
{{\lambda}^{(0)}_i}^{-1} & 0\\
0 & \lambda^{(0)}_i
\end{matrix}
\right),
\hspace{30pt}
\overline{\rho}(\eta)=
\left(
\begin{matrix}
0 & -1\\
1 & 0
\end{matrix}
\right)
$$
hold. 
Therefore 
$
\rho(\eta \tau^{0}\tilde{x}_i \eta^{-1})=
\left(
\begin{matrix}
 {{\lambda}^{(0)}_i}^{-1} & 0\\
0 & \lambda^{(0)}_i
\end{matrix}
\right),
$
and this is $\overline{\rho}(\tau^{0}\tilde{x}_i)$.
\end{proof}

\begin{proof}[Proof\ of {\rm\ Theorem~\ref{Main}}]
We decompose the proof into two cases: (1) $m$ is even or (2) $m$ is odd. 
 In case (1),  by Lemma~\ref{alt}, we only need to count the number of irreducible $SL(2,\C)$-metabelian representations of $G$ in the Lemma.
Since the presentation in (\ref{red}) is exactly of the same form as the Lin's presentation~(\ref{L1}),
each $\lambda^{(0)}_1,\ldots, \lambda^{(0)}_{2g}$ must satisfy the following equations as explained in~\cite{L}:
\begin{equation}\label{eq}
\mathrm{e}^{\sqrt{-1}}{r_1}^{\omega_{i1}\theta_i} \cdots {r_{2g}}^{\omega_{i2g}\theta_{2g}} = 1,
\end{equation}
\begin{equation}\label{eq2}
 {r_1}^{\omega_{i1}}\cdots {r_{2g}}^{\omega_{i2g}} = 1,
\end{equation}
where $\lambda^{(0)}_i = r_i\mathrm{e}^{\sqrt{-1}\theta_i}$ for $i = 1,\ldots ,2g$ and the matrix $(\omega_{ij}) = V + {^t}V$  is defined in~Section 2.3.
With some linear algebra, one can see that
the solution of~(\ref{eq2}) is $(r_1,\ldots,r_{2g}) = (1,\ldots,1)$ and
 the number of non-trivial solutions of~(\ref{eq}) is $|{\det}(V + {^t}V)|-1 = |\Delta_K(-1)|-1$. 
If $\{\gamma_i\}_{0\leq i \leq 2g}$ is a solution of~(\ref{eq}) and~(\ref{eq2}), then $\{{\gamma_i}^{-1}\}_{0\leq i \leq 2g}$ is also, which is given by the conjugation of $\rho$.
Therefore the number of irreducible $SL(2,\C)$-metabelian representations of $\pKmn$ is $\frac{|\Delta_K(-1)|-1}{2}$.

In case (2), 
if $q$ is even, then $\lambda^{(j)}_i = \lambda^{(0)}_i$ for all $j$ by~(\ref{case}).
Then the relation $\eta \tau^{j+n}\tilde{x}_i \eta^{-1} = \tau^{j}\tilde{x}_i$ in~(\ref{R2}) gives 
$\lambda^{(j)}_i = {{\lambda}^{(j)}_i}^{-1}$ for all $i,j$.
In this case, there is no irreducible $SL(2,\C)$-metabelian representation of $\pKmn$ by Lemma~\ref{con}.
Suppose that $q$ is odd.
From the relations~(\ref{rel}), we have
\begin{equation*}
\tau^{j}\tilde{x}_i = \eta^{q} \tau^{j +1}\tilde{x}_i \eta^{-q} = \eta^{mq} \tau^{j+m}\tilde{x}_i \eta^{-mq}
= \eta^{mq} \tau^{j}\tilde{x}_i \eta^{-mq}.
\end{equation*}
Then we have $\lambda^{(j)}_i = {{\lambda}^{(j)}_i}^{-1}$ for all $i,j$ since $mq$ is odd, 
and hence there is no irreducible $SL(2,\C)$-metabelian representation of $\pKmn$ by Lemma~\ref{con}.
Thus the assertion holds.
\end{proof}

\end{document}